\documentclass{amsart}
\usepackage{color}
\usepackage{amsmath}
\usepackage{amssymb}
\usepackage{amsthm}

\newtheorem{dfn}{Definition }[section]

\newtheorem{thm}[dfn]{Theorem}

\newlength{\fixboxwidth}
\newtheorem{lem}[dfn]{Lemma}
\theoremstyle{remark}
\newtheorem{rem}[dfn]{Remark}

\newcommand{\mgk}{\mathcal{W}^\alpha_{\alpha_1,\alpha_2}(\Rn)}
\newcommand{\cpu}{c_\infty(1/p,1/u)}
\newcommand{\R}{{\mathbb{R}}}
\newcommand{\Rn}{{\mathbb{R}^n}}
\newcommand{\N}{\mathbb{N}}
\newcommand{\Z}{\mathbb{Z}}

\newcommand{\calP}{\mathcal{P}(\Rn)}
\newcommand{\Plog}{\mathcal{P}^{\log}(\Rn)}
\newcommand{\q}{{q(\cdot)}}
\newcommand{\p}{{p(\cdot)}}
\renewcommand{\u}{{u(\cdot)}}
\newcommand{\SSn}{\mathcal{S}'(\Rn)}
\newcommand{\Sn}{\mathcal{S}(\Rn)}
\newcommand{\Lp}{L_\p(\Rn)}
\newcommand{\Mup}{M_{p(\cdot)}^{u(\cdot)}(\Rn)}
\newcommand{\Muplq}{M_{p(\cdot)}^{u(\cdot)}(\ell_\q)}

\newcommand{\vek}[1]{\boldsymbol{#1}}

\newcommand{\norm}[2]{\left\|\left.{#1}\right|{#2}\right\|}

\newcommand{\Mufwpqx}{\mathcal{E}^{\vek{w},u(\cdot)}_{\p,\q}(\Rn)}
\newcommand{\Fwpxqx}{F^{\vek{w}}_{\p,\q}(\Rn)}
\newcommand{\supp}{\operatorname{supp}}

\begin{document}

\title{Variable exponent Triebel-Lizorkin-Morrey spaces}

\author[A. Caetano]{Ant\'{o}nio Caetano}
\address{Center for R\&D in Mathematics and Applications, Department of Mathematics, University of Aveiro, 3810-193 Aveiro, Portugal}
\email{acaetano@ua.pt}

\author[H. Kempka]{Henning Kempka}
\address{Department of Fundamental Sciences, PF 100314, University of Applied Sciences Jena, 07703 Jena, Germany}
\email{henning.kempka@eah-jena.de}

\thanks{This research was partially supported by the project \emph{Smoothness Morrey spaces with variable exponents} approved under the agreement Projektbezogener Personenaustausch mit Portugal -- A{\c{c}}{\~o}es Integradas Luso-Alem{\~a}s' / DAAD-CRUP. It was also supported through CIDMA (Center for Research and Development in Mathematics and Applications) and FCT (Foundation for Science and Technology) within project UID/MAT/04106/2019.}
\thanks{\copyright 2019. Licensed under the CC BY-NC-ND 4.0 license http://creativecommons.org/licenses/by-nc-nd/4.0/}

\date{\today}



\begin{abstract}
We introduce variable exponent versions of Morreyfied Triebel-Lizorkin spaces. To that end, we prove an important convolution inequality which is a replacement for the Hardy-Littlewood maximal inequality in the fully variable setting.\\Using it we obtain characterizations by means of Peetre maximal functions and use them to show the independence of the introduced spaces from the admissible system used.
\end{abstract}

\maketitle

\section{Introduction}
In this paper we introduce and present some properties of spaces which mix two recent trends in the literature: starting from the Besov and Triebel-Lizorkin spaces $B^s_{p,q}(\Rn)$ and $F^s_{p,q}(\Rn)$,
\renewcommand{\theenumi}{\Roman{enumi}}%
\begin{enumerate}
	\item on one hand, one \emph{Morreyfies} them in some way, that is, replace the $L_p(\Rn)$ spaces in their construction by Morrey spaces $M^u_p(\Rn)$;
	\item on the other hand, one makes the parameters $s$, $p$, $q$ and $u$ variable.
\end{enumerate}
The first trend started with Kozono and Yamazaki \cite{KY94} for the Besov scale in connection with the study of Navier-Stokes equations, giving rise to spaces which we denote by ${\mathcal N}^{s,u}_{p,q}(\Rn)$. The corresponding Triebel-Lizorkin scale with spaces denoted here by ${\mathcal E}^{s,u}_{p,q}(\Rn)$ was introduced by Tang and Xu \cite{TangXu}, with Sobolev-Morrey spaces being already considered by S. Campanato \cite{Cam63} as early as 1963. Related so-called Besov-type and Triebel-Lizorkin-type spaces $B^{s,\tau}_{p,q}(\Rn)$ and $F^{s,\tau}_{p,q}(\Rn)$ were introduced by El Baraka \cite{ElBar02, ElBar06}, at least in the Banach case, and Yang and Yuan \cite{YangYuan08,YangYuan10}, at least in the homogeneous case. For more details on these spaces and corresponding literature see \cite{YSY10}, the surveys \cite{Sic12,Sic13} and also \cite{Rosenthal}. Be aware, as stressed in \cite{Sic12,Sic13}, that the order in which the parameters show in the notation for the $\mathcal N$ and $\mathcal E$ spaces above can change with the authors. Also, usually only the smoothness $s$ is written as superscript, so here we add a little bit to the chaos by moving another parameter to superscript. But our idea is the following: as subscripts we have only left the parameters with the roles corresponding to $p$ and $q$ in the --- let us say --- standard Besov and Triebel-Lizorkin scales; as superscripts we start by writing the smoothness parameter $s$ and then join in the new parameter. We followed the same philosophy as for the Besov-type and Triebel-Lizorkin-type spaces, though the roles of the new parameters $\tau$ and $u$ differ.

The second trend is more difficult to trace, so we opt to mention \cite{DHR,AH10} --- where the spaces $F^{s(\cdot)}_{\p,\q}(\Rn)$ and $B^{s(\cdot)}_{\p,\q}(\Rn)$, with all the three parameters variable were introduced --- and refer to those papers for some history on the subject and its connection with the study of PDEs. These spaces are build up for variable $\p$ on the variable Lebesgue spaces $L_\p$, which can be traced back to Orlicz \cite{Orl31} and Kov\'{a}{\v c}ik and R\'{a}kosn\'{\i}k \cite{Kovacik}.

Since our results in this paper are only concerned with a Morreyfication of the variable Triebel-Lizorkin scale, from now on we will omit references to possible ways of Morreyfying the variable Besov scale. As in the case of $B^{s(\cdot)}_{\p,\q}(\Rn)$, this poses different challenges, see \cite{AC18}.

Once Morrey spaces $\Mup$ with variable exponents were introduced by Almeida, Hasanov and Samko \cite{AHS08}, Kokilashvili and Meskhi \cite{KM08} and  Mizuta and Shimomura \cite{MS08}, the way was open to try to mix the two trends mentioned above. A first step was done by Fu and Xu \cite{FuXu} by considering Triebel-Lizorkin-Morrey spaces ${\mathcal E}^{s,\u}_{\p,q}(\Rn)$ --- so, still keeping $s$ and $q$ fixed. However, there are some flaws in the \emph{proofs} of some results in that paper. We point out one in Remark \ref{rem:FuXu} below, but for a more detailed analysis see \cite[Remarks 3.12 and 4.20]{Rosenthal}. Later Ho \cite{Ho12} considered ${\mathcal E}^{s(\cdot),\u}_{\p,\q}(\Rn)$ as a particular case of a more general framework. Although Ho considered all four parameters variable, he has imposed on the other hand some not so natural technical restrictions on $s$, $p$ and $u$, e.g. 
\begin{equation}
\sup_{x \in \Rn}\Big( \frac{1}{p(x)}-\frac{1}{u(x)}\Big) < \frac{1}{\sup p}
\label{eq:Horest}
\end{equation}
(see \cite[Definition 6.6]{Ho12}). Variable exponent versions of the Triebel-Lizorkin-type spaces $F^{s,\tau}_{p,q}(\Rn)$ were also recently considered by Yang, Yuan and Zhuo \cite{YYZ15} and Drihem \cite{Dri}. It is not clear how they compare with each other, but the former is closer to our intentions here. Actually, Yang, Yuan and Zhuo replaced the $\tau(\cdot)$ by some more general structure, but under suitable conditions, in particular \eqref{eq:Horest} above (see \cite[Theorem 3.12(i)]{YYZ15}), their spaces cover the Triebel-Lizorkin-Morrey spaces ${\mathcal E}^{s(\cdot),\u}_{\p,\q}(\Rn)$.

In the present paper we get rid of the restriction \eqref{eq:Horest}, our other conditions on $p$ and $u$ are quite natural (for $u$ we don't even impose any smoothness property) and we are able to prove characterization of the spaces by Peetre maximal functions. Actually, we even consider 2-microlocal versions $\Mufwpqx$, building on the knowledge we have acquired in \cite{Kempka09} and \cite{AC16} in the non-Morrey situation. After this work was essentially complete, we learned that 2-microlocal versions of the variable exponent Triebel-Lizorkin-type spaces mentioned above were also recently considered in \cite{WYYZ18}. However, in the part that covers our 2-microlocal variable exponent Triebel-Lizorkin-Morrey spaces the undesirable restrictions mentioned above are still assumed, see \cite[Theorem 4.7(iii)]{WYYZ18}.

Our main results are the convolution inequality in Theorem \ref{thm:MorreyHardy} and the Peetre maximal function characterization of $\Mufwpqx$ in Theorem \ref{thm:lm}.

In a forthcoming paper \cite{PartII} we also use the obtained results to present atomic and molecular representations for these spaces.

We have already advised the reader to check carefully the notation before comparing results coming from different sources. In the comparison made above we have translated everything into our notation. The reader should also be aware that even the very notion of Morrey space with variable exponents, and not only its notation, might differ from work to work.

\section{Preliminaries}
\subsection{General notation}
Here we introduce some of the general notation we use throughout the paper. $\N$, $\N_0$, $\Z$, $\R$ and $\mathbb C$ have the usual meaning as sets of numbers, as well as their $n$-th powers for $n \in \N$. The Euclidean norm in $\Rn$ is denoted by $|\cdot|$, though this notation is also used for the norm of a multi-index, and for Lebesgue measure when it is being applied to (measurable) subsets of $\Rn$. 

The symbol $\Sn$ stands for the usual Schwartz space of infinitely differentiable rapidly decreasing complex-valued functions on $\Rn$.
By { }$\hat{\phi}$ we denote the Fourier transform of $\phi \in \Sn$ in the version
$$ \hat{\phi}(x):=\frac{1}{(2\pi)^{n/2}}\int_{\Rn} e^{-ix\cdot\xi} \phi(\xi)\, d\xi, \quad x \in \Rn,$$
and by ${\phi}^\vee$ we denote the inverse Fourier transform of $\phi$. These transforms are topological isomorphisms in $\Sn$ which extend in the usual way to the space $\SSn$ of tempered distributions, the dual space of $\Sn$.

For two complex or extended real-valued measurable functions $f,g$ on $\Rn$ the convolution $f*g$ is given, whenever it makes sense a.e., by
\begin{align*}
(f\ast g)(x):=\int_\Rn f(x-y)g(y)dy, \quad\text{for } x \in \Rn.
\end{align*}
By $c,c_1,c_\phi,...>0$ we denote constants which may change their value from one line to another. Further, $f\lesssim g$ means that there exists a constant $c>0$ such that $f\leq cg$ holds for a set of variables on which $f$ and $g$ may depend on and which shall be clear from the context. If we write $f\approx g$ then there exists constants $c_1,c_2>0$ with $c_1f\leq g\leq c_2f$. And we shall then say that the expressions $f$ and $g$ are equivalent (across the considered set of variables).

By $Q_r(x) \subset \Rn$ we denote the open cube in $\Rn$ with center $x\in\Rn$ and sides parallel to the axes and of length $2r>0$. Further, $B_r(x) \subset \Rn$ is the open ball in $\Rn$ with center $x\in\Rn$ and radius $r>0$ and by $\chi_A$ we denote the characteristic function of any subset $A$ of $\Rn$.

Given topological vector spaces $A$ and $B$, the notation $A \hookrightarrow B$ will be used to mean that the space $A$ is continuously embedded into the space $B$.

\subsection{Variable exponent Lebesgue spaces}
The set of variable exponents $\mathcal{P}(\Rn)$ is the collection of all measurable functions $p:\Rn\to(0,\infty]$ with $p^-:=\operatornamewithlimits{ess-inf}_{x\in\Rn}p(x)>0$. Further, we set $p^+:=\operatornamewithlimits{ess-sup}_{x\in\Rn}p(x)$.
For exponents with $p(x)\geq1$ and complex or extended real-valued measurable functions $f$ on $\R^n$ a semimodular is defined by
\begin{equation}\notag
\varrho_\p(f):=\int_\Rn \phi_{p(x)}(|f(x)|)\,dx,
\end{equation}
where
$$
\phi_{p(x)}(t) :=
\begin{cases}
t^{p(x)} & \text{ if } p(x)\in (0,\infty), \\
0 & \text{ if } p(x)=\infty \text{ and } t\in [0,1], \\
\infty & \text{ if } p(x)=\infty \text{ and } t\in(1,\infty], \\
\end{cases}
$$
and the variable exponent Lebesgue space $\Lp$ is given by
$$\Lp:=\{f: \text{ there exists a $\lambda>0$ with }\varrho_\p\left(f/\lambda\right)<\infty\},$$
with their elements being taken in the usual sense of equivalence classes of a.e. coincident functions.
This space is complete and normed, hence a Banach space, with the norm
\begin{align*}
\norm{f}{\Lp}:=\inf\{\lambda>0:\varrho_\p(f/\lambda)\leq 1\}.
\end{align*}
It shares many properties with the usual Lebesgue spaces, see for a wide overview \cite{Kovacik}, \cite{DHHR}, \cite{CruzUribe}, but there are also some differences, e.g. it is not translation invariant. By the property 
\begin{equation}\label{eq:Lp/t}
\norm{f}{\Lp}=\norm{|f|^t}{L_{\frac\p t}(\Rn)}^{1/t}\quad\text{for any $t>0$}
\end{equation}
it is also possible to extend the definition of the spaces $\Lp$ to all exponents $p\in\mathcal{P}(\Rn)$. In such more general setting the functional $\norm{\cdot\,}{\Lp}$ need not be a norm, although it is always a quasi-norm.\\
Many theorems for variable Lebesgue spaces $\Lp$ are only valid for exponents $\p$ within a subclass of $\mathcal{P}(\Rn)$ where they satisfy certain regularity conditions. An appropriate subclass in this sense is the set $\Plog$ defined below.
\begin{dfn}
 Let $g:\Rn\to\R$.
 \begin{itemize}
  \item[(i)] We say that $g$ is locally $\log$ H\"older continuous, $g\in C^{\log}_{{\rm loc}}(\Rn)$, if there exists a constant $c_{\log}(g) > 0$ with
  \begin{align*}
   |g(x)-g(y)|\leq\frac{c_{\log}(g)}{\log(e+\frac{1}{|x-y|})}\quad\text{for all }x,y\in\Rn.
  \end{align*}
  \item[(ii)] We say that $g$ is globally $\log$ H\"older continuous, $g\in C^{\log}$, if it is locally $\log$ H\"older continuous and there exist a $g_\infty\in \R$ and a constant $c_\infty(g) > 0$ with
  \begin{align}\label{dfn:cinfty}
   |g(x)-g_\infty|\leq\frac{c_\infty(g)}{\log(e+|x|)}\quad\text{for all }x\in\Rn.
  \end{align}
  \item[(iii)] We write $g\in\Plog$ if $0<g^-\leq g(x)\leq g^+\leq\infty$ with $1/g\in C^{\log}(\Rn)$.
 \end{itemize}
\end{dfn}
Since a control of the quasi-norms of characteristic functions of balls in variable exponent spaces will be crucial for our estimates, we present below a result in that direction and which is an adapted version of \cite[Corollary 4.5.9]{DHHR} to the case $0<p^-\leq p^+\leq\infty$. It can easily be obtained from that result by exploring the property \eqref{eq:Lp/t} above.
\begin{lem}\label{lem:xinorm}
 Let $p\in\Plog$. Then for all $x_0\in\R^n$ and all $r>0$ we have that 
 \begin{align*}
  \norm{\chi_{B_r(x_0)}}{L_{p(\cdot)}(\Rn)}&\approx \norm{\chi_{Q_r(x_0)}}{L_{p(\cdot)}(\Rn)}\\
  &\approx \begin{cases}
            r^{\frac{n}{p(x)}}\quad&,\;\text{if }r\leq 1\text{ and }x\in B_r(x_0)\\
            r^{\frac{n}{p_\infty}}&,\;\text{if }r\geq1
           \end{cases}.
 \end{align*}
Here we denote $\frac1{p_\infty}:=\Big(\frac1{p}\Big)_\infty$ which is given by \eqref{dfn:cinfty}.
\end{lem}
\section{Variable exponents Morrey spaces}

Now, we can define the Morrey spaces which we are interested in. 
\begin{dfn}
Let $p,u\in\calP$ with $p\leq u$. Then the Morrey space $\Mup$ is the collection of all (complex or extended real-valued) measurable functions $f$ on $\Rn$ with (quasi-norm given by)
\begin{align*}
 \norm{f}{\Mup}:=\sup_{x\in\Rn,r>0}r^{n\left(\frac1{u(x)}-\frac1{p(x)}\right)}\norm{f}{L_{p(\cdot)}(B_r(x))}<\infty.
\end{align*} 
\end{dfn}

Next we state the convolution inequality from \cite{DHR}, which is heavily used as a replacement of the Hardy-Littlewood maximal inequality in the variable setting, where $L_{p(\cdot)}(\ell_{q(\cdot)})$ stands for the set of all sequences $(f_\nu)_{\nu \in \N_0}$ of (complex or extended real-valued) measurable functions on $\Rn$ such that $\norm{\left(\sum_{\nu=0}^\infty|f_\nu(\cdot)|^{q(\cdot)}\right)^{1/q(\cdot)}}{L_{p(\cdot)}(\Rn)}$ is finite.\\
\begin{lem}\label{lem:Convinequality}
 Let $\eta_{\nu,m}(x):=2^{\nu n}(1+2^\nu|x|)^{-m}$ for $\nu\in\N_0$ and $m>0$. Let $p,q\in\Plog$ with $1<p^-\leq p^+<\infty$ and $1<q^-\leq q^+<\infty$.
 Then for every $m>n$ there exists a constant $c>0$ such that for all sequences $(f_\nu)_\nu\in L_{p(\cdot)}(\ell_{q(\cdot)})$
 \begin{align*}
  \norm{\left(\sum_{\nu=0}^\infty|\eta_{\nu,m}\ast f_\nu(\cdot)|^{q(\cdot)}\right)^{1/q(\cdot)}}{L_{p(\cdot)}(\Rn)}\leq c\norm{\left(\sum_{\nu=0}^\infty|f_\nu(\cdot)|^{q(\cdot)}\right)^{1/q(\cdot)}}{L_{p(\cdot)}(\Rn)}.
 \end{align*}
\end{lem}
Now, we are able to formulate and prove the corresponding result with Morrey spaces instead of Lebesgue spaces, where $M_{p(\cdot)}^{u(\cdot)}(\ell_{\q})$ shall stand for the set of all sequences $(f_\nu)_{\nu \in \N_0}$ of (complex or extended real-valued) measurable functions on $\Rn$ such that $\norm{\left(\sum_{\nu=0}^\infty|f_\nu(\cdot)|^{\q}\right)^{1/\q}}{\Mup}$ is finite.\\
Such a result will be the main tool for these variable spaces and will be used in further results to be presented in this paper. 
\begin{thm}\label{thm:MorreyHardy}
Let $\eta_{\nu,m}$ be as in the preceding lemma. Let  $p,q\in\Plog$ and $u\in\mathcal{P}(\Rn)$ with $1<p^-\leq p(x)\leq u(x)\leq \sup u < \infty$ and  $q^-,q^+\in(1,\infty)$. For every 
\begin{align}
\label{condition_m}
m>n+n\max\left(0,\sup_{x\in\Rn}\left(\frac1{p(x)}-\frac1{u(x)}\right)-\frac1{p_\infty}\right)
\end{align} 
there exists a $c>0$ such that for all $(f_\nu)_\nu \subset M_{p(\cdot)}^{u(\cdot)}(\ell_{\q})$
 \begin{align*}
  \norm{\left(\sum_{\nu=0}^\infty|\eta_{\nu,m}\ast f_\nu(\cdot)|^{q(\cdot)}\right)^{1/\q}}{\Mup}\leq c\norm{\left(\sum_{\nu=0}^\infty|f_\nu(\cdot)|^{\q}\right)^{1/\q}}{\Mup}.
 \end{align*}
\end{thm}
\begin{proof}
 
 \underline{First step:} We take an arbitrary $x_0\in\Rn$ and $r>0$ and decompose for every $\nu\in\N_0$
 \begin{align*}
  f_\nu(x)=f_\nu^0(x)+\sum_{i=1}^\infty f_\nu^i(x),
 \end{align*}
 where $f_\nu^0:=f_\nu\chi_{B_{2r}(x_0)}$ and $f_\nu^i:=f_\nu\chi_{B_{2^{i+1}r}(x_0)\setminus B_{2^{i}r}(x_0)}$. Now, we have using triangle inequalities
 \begin{align}
   &\norm{\left(\sum_{\nu=0}^\infty|\eta_{\nu,m}\ast f_\nu(\cdot)|^{\q}\right)^{1/\q}}{L_{p(\cdot)}(B_r(x_0))} \label{vectLpconvol} \\
	&\leq \norm{\left(\sum_{\nu=0}^\infty\Big(\eta_{\nu,m}\ast \sum_{i=0}^\infty |f_\nu^i(\cdot)|\Big)^{\q}\right)^{1/\q}}{L_{p(\cdot)}(B_r(x_0))}\notag\\
		&\leq   \norm{\sum_{i=0}^\infty \left(\sum_{\nu=0}^\infty\big(\eta_{\nu,m}\ast |f_\nu^i(\cdot)|\big)^{\q}\right)^{1/\q}}{L_{p(\cdot)}(B_r(x_0))}
   \leq I+I\!I,\notag
 \end{align}
\noindent where
$$I := \norm{\left(\sum_{\nu=0}^\infty\big(\eta_{\nu,m}\ast |f_\nu^0(\cdot)|\big)^{\q}\right)^{1/\q}}{L_{p(\cdot)}(B_r(x_0))},$$
$$I\!I := \norm{\sum_{i=1}^{\infty} \left(\sum_{\nu=0}^\infty\big(\eta_{\nu,m}\ast |f_\nu^i(\cdot)|\big)^{\q}\right)^{1/\q}}{L_{p(\cdot)}(B_r(x_0))}. $$
We shall show, under the given hypotheses, that both terms I and II above can be estimated by
\begin{equation}
\leq c\, r^{n\left(\frac1{p(x_0)}-\frac1{u(x_0)}\right)} \norm{\left(\sum_{\nu=0}^\infty|f_\nu(\cdot)|^{\q}\right)^{1/\q}}{\Mup}
\label{eq:mainestimate}
\end{equation}
\noindent with $c>0$ independent of the $(f_\nu)_\nu$, $x_0$ and $r$ considered. So that dividing (\ref{vectLpconvol}) by $r^{n\left(\frac1{p(x_0)}-\frac1{u(x_0)}\right)}$ and taking the supremum over all $x_0\in\Rn$ and $r>0$ gives the desired inequality and finishes the proof.

\smallskip 

 \underline{Second step:} We get the estimate \eqref{eq:mainestimate} for $I$ with the help of Lemma \ref{lem:Convinequality} (using the hypothesis $m>n$):
 \begin{align}
  & I \leq \norm{\left(\sum_{\nu=0}^\infty \big(\eta_{\nu,m}\ast |f_\nu^0(\cdot)|\big)^{\q}\right)^{1/\q}}{L_{p(\cdot)}(\Rn)}\lesssim\norm{\left(\sum_{\nu=0}^\infty|f_\nu^0(\cdot)|^{\q}\right)^{1/\q}}{L_{p(\cdot)}(\Rn)}\notag\\
  &\hspace{1em}=(2r)^{n\left(\frac1{p(x_0)}-\frac1{u(x_0)}\right)} (2r)^{n\left(\frac1{u(x_0)}-\frac1{p(x_0)}\right)} \norm{\left(\sum_{\nu=0}^\infty|f_\nu(\cdot)|^{\q}\right)^{1/\q}}{L_{p(\cdot)}(B_{2r}(x_0))}\notag\\
  &\hspace{1em}\leq 2^{\frac{n}{p^-}} r^{n\left(\frac1{p(x_0)}-\frac1{u(x_0)}\right)}\norm{\left(\sum_{\nu=0}^\infty|f_\nu(\cdot)|^{\q}\right)^{1/\q}}{\Mup}.\notag
 \end{align}

\smallskip

\underline{Third step:} First, we show a size estimate to tackle $I\!I$. We observe that for $x\in B_{r}(x_0)$ and $y\in {B_{2^{i+1}r}(x_0)\setminus B_{2^{i}r}(x_0)}$ we have the following inequality:
\begin{align*}
 |x-y|&\geq||x-x_0|-|y-x_0||\geq|y-x_0|-|x-x_0|\geq 2^ir-r\geq c 2^{i}r.
\end{align*}
Now, we can use this and estimate with $x\in B_r(x_0)$ for every $i\in\N$ \vspace{-.1cm}
\begin{align}
\left(\sum_{\nu=0}^\infty \big(\eta_{\nu,m}\ast |f_\nu^i(x)|\big)^{q(x)}\right)^{\frac1{q(x)}}
&\lesssim \left(\sum_{\nu=0}^\infty \int_\Rn 2^{\nu n}(1+2^\nu2^ir)^{-m}|f_\nu^i(y)|\, dy\right)\notag\\
&\hspace{-8em}=\int_{B_{2^{i+1}r}(x_0)}\left(\sum_{\nu=0}^\infty2^{\nu n}(1+2^\nu2^ir)^{-m}\left|f_\nu^i(y)\right|\right)dy,\label{eq:SE}
\end{align}

where we have used $\ell_1\hookrightarrow\ell_{q(x)}$  and Beppo-Levi's theorem.

\smallskip

\underline{Forth step:} We use \eqref{eq:SE} to handle $I\!I$. By applying H\"older's inequality in the integral with $p(\cdot) > 1$ we get
\begin{align}
& I\!I \lesssim \norm{\chi_{B_r(x_0)}}{L_{p(\cdot)}(\Rn)} \sum_{i=1}^\infty \int_{B_{2^{i+1}r}(x_0)}\left(\sum_{\nu=0}^\infty2^{\nu n}(1+2^\nu2^ir)^{-m}\left|f_\nu^i(y)\right|\right)dy\notag\\
&\lesssim \norm{\chi_{B_r(x_0)}}{L_{p(\cdot)}(\Rn)} \sum_{i=1}^\infty \norm{\chi_{B_{2^{i+1}r}(x_0)}}{L_{p'(\cdot)}(\Rn)}\notag\\
&\hspace{9em}\times\norm{\left(\sum_{\nu=0}^\infty2^{\nu n}(1+2^\nu2^ir)^{-m}\left|f_\nu^i(\cdot)\right|\right)}{L_{p(\cdot)}(\Rn)}.\label{eq:new1}
\end{align}
We estimate the last norm with H\"older's inequality in the sum with $\q > 1$ and use afterwards $\ell_1\hookrightarrow\ell_{q'(\cdot)}$ to obtain
\begin{align}
&\norm{\left(\sum_{\nu=0}^\infty2^{\nu n}(1+2^\nu2^ir)^{-m}\left|f_\nu^i(\cdot)\right|\right)}{L_{p(\cdot)}(\Rn)}\notag\\
&\hspace{2em}\leq\norm{\left(\sum_{\nu=0}^\infty2^{\nu nq'(\cdot)}(1+2^\nu2^ir)^{-mq'(\cdot)}\right)^{1/q'(\cdot)}\left(\sum_{\nu=0}^\infty\left|f_\nu^i(\cdot)\right|^\q\right)^{1/\q}}{L_{p(\cdot)}(\Rn)}\notag\\
&\hspace{2em}\leq\left(\sum_{\nu=0}^\infty2^{\nu n}(1+2^\nu2^ir)^{-m}\right)\norm{\left(\sum_{\nu=0}^\infty\left|f_\nu^i(\cdot)\right|^\q\right)^{1/\q}}{L_{p(\cdot)}(\Rn)}\notag
\\
&\hspace{2em}\leq\left(\sum_{\nu=0}^\infty2^{\nu n}(1+2^\nu2^ir)^{-m}\right)(2^{i+1}r)^{\left(\frac{n}{p(x_0)}-\frac{n}{u(x_0)}\right)}\notag
\\
&\qquad\qquad \times (2^{i+1}r)^{\left(\frac{n}{u(x_0)}-\frac{n}{p(x_0)}\right)}\norm{\left(\sum_{\nu=0}^\infty\left|f_\nu(\cdot)\right|^\q\right)^{1/\q}}{L_{p(\cdot)}(B_{2^{i+1}r}(x_0))}\notag
\end{align}
\begin{align}
&\hspace{2em}\leq\left(\sum_{\nu=0}^\infty2^{\nu n}(1+2^\nu2^ir)^{-m}\right)(2^{i+1}r)^{\left(\frac{n}{p(x_0)}-\frac{n}{u(x_0)}\right)}\notag\\
&\hspace{13em}\times\norm{\left(\sum_{\nu=0}^\infty\left|f_\nu(\cdot)\right|^\q\right)^{1/\q}}{\Mup}.
\label{eq:new2}
\end{align}
Therefore, connecting \eqref{eq:new1} and \eqref{eq:new2}, we obtain \eqref{eq:mainestimate} for the term $I\!I$ if we can show that
\begin{align}
\label{eq:final}
&I\!I\!I := \norm{\chi_{B_r(x_0)}}{L_{p(\cdot)}(\Rn)} \sum_{i=1}^\infty \norm{\chi_{B_{2^{i+1}r}(x_0)}}{L_{p'(\cdot)}(\Rn)}\notag\\
&\hspace{2em}\times \left(\sum_{\nu=0}^\infty2^{\nu n}(1+2^\nu2^ir)^{-m}\right)(2^{i+1}r)^{\left(\frac{n}{p(x_0)}-\frac{n}{u(x_0)}\right)}
\leq c\, r^{\left(\frac{n}{p(x_0)}-\frac{n}{u(x_0)}\right)},
\end{align}
with $c>0$ independent of $r>0$ and $x_0\in\Rn$.\\

In order to show \eqref{eq:final} we concentrate on two cases.\\
\underline{Case $0<r<1$:} We choose $J\in \N$ so that $2^Jr \leq 1 < 2^{J+1}r$ if such a $J$ exists, choose $J=1$ otherwise and split the summation with respect to $i\in\N$ in two sums $\sum_{i=1}^{J-1}$ and $\sum_{i=J}^{\infty}$. Of course, when $J=1$ the sum $\sum_{i=1}^{J-1}$ does not exist. Denote by $I\!I\!I_1$ and $I\!I\!I_2$ the corresponding splitting of $I\!I\!I$. We obtain with Lemma \ref{lem:xinorm} that
\begin{align}
&I\!I\!I_1 \lesssim r^{\frac n{p(x_0)}} \sum_{i=1}^{J-1}(2^{i+1}r)^{\frac n{p'(x_0)}}(2^{i+1}r)^{n\left(\frac1{p(x_0)}-\frac1{u(x_0)}\right)}\left(\sum_{\nu=0}^\infty2^{\nu n}(1+2^\nu2^ir)^{-m}\right)\notag\\
&\hspace{2em}=r^{n\left(\frac1{p(x_0)}-\frac1{u(x_0)}\right)}\sum_{i=1}^{J-1}2^{-(i+1)\frac n{u(x_0)}}(2^{i+1}r)^{n\left(\frac1{p(x_0)}+\frac1{p'(x_0)}\right)}\notag\\
&\hspace{17em}\times\left(\sum_{\nu=0}^\infty2^{\nu n}(1+2^\nu2^ir)^{-m}\right)\notag\\
&\hspace{2em}\lesssim r^{n\left(\frac1{p(x_0)}-\frac1{u(x_0)}\right)}\sum_{i=1}^{J-1}2^{-i\frac n{\sup u}}(2^{i}r)^{n}\left(\sum_{\nu=0}^\infty2^{\nu n}(1+2^\nu2^ir)^{-m}\right)\label{eq:new3}.
\end{align} 
Now, we look at the last term in \eqref{eq:new3}. We denote by $V\in\N$ the unique number with $2^V\leq(2^ir)^{-1}<2^{V+1}$ and get with $m>n$
\begin{align*}
\left(\sum_{\nu=0}^\infty2^{\nu n}(1+2^\nu2^ir)^{-m}\right)
&=\sum_{\nu=0}^V2^{\nu n}(1+2^\nu2^ir)^{-m}+\sum_{\nu=V+1}^\infty2^{\nu n}(1+2^\nu2^ir)^{-m}\\
&\hspace{-1em}\leq \sum_{\nu=0}^V2^{\nu n}+\sum_{\nu=V+1}^\infty2^{\nu (n-m)}(2^ir)^{-m}
\\
&\hspace{-1em}=2^{Vn}\sum_{\nu=0}^V2^{-(V-\nu) n}+(2^ir)^{-m}\sum_{\nu=V+1}^\infty2^{\nu (n-m)}
\end{align*}
\begin{align*}
&\hspace{-1em}\leq(2^ir)^{-n}\sum_{k=0}^\infty2^{-kn}+(2^ir)^{-m}2^{(V+1)(n-m)}\sum_{k=0}^\infty2^{-k(m-n)}\\
&\hspace{-1em}\leq(2^ir)^{-n}\frac1{1-2^{-n}}+(2^ir)^{-m}(2^ir)^{m-n}\frac1{1-2^{n-m}}\\
&\hspace{-1em}=c\,(2^ir)^{-n},
\end{align*}
where $c>0$ only depends on $m$ and $n$. Plugging this now into \eqref{eq:new3} we obtain \eqref{eq:final} for the term $I\!I\!I_1$ of $I\!I\!I$.\\
Now, we handle the term $I\!I\!I_2$, where it holds $2^{i+1}r>1$ for $i\geq J$. We get, using Lemma \ref{lem:xinorm}, the estimate $2^{\nu n}(1+2^\nu2^ir)^{-m} \leq 2^{\nu(n-m)}(2^ir)^{-m}$ and the hypothesis \eqref{condition_m} on $m$, that
\begin{align}
&I\!I\!I_2 \lesssim r^{n\left(\frac1{p(x_0)}-\frac1{u(x_0)}\right)} \sum_{i=J}^{\infty} r^{\frac n{p(x_0)}}(2^{i+1}r)^{\frac n {p'_\infty}}(2^{i+1})^{n\left(\frac1{p(x_0)}-\frac1{u(x_0)}\right)}(2^ir)^{-m}\notag\\
&\hspace{2em}\approx r^{n\left(\frac1{p(x_0)}-\frac1{u(x_0)}\right)}r^{\left(\frac n{p'_\infty}+\frac n{p(x_0)}-m\right)} \sum_{i=J}^{\infty} 2^{(i+1)\left(\frac n{p'_\infty}+\frac n{p(x_0)}-\frac n{u(x_0)}-m\right)}\notag\\
&\hspace{2em}= r^{n\left(\frac1{p(x_0)}-\frac1{u(x_0)}\right)}r^{\left(n-\frac n{p_\infty}+\frac n{p(x_0)}-m\right)} \sum_{i=J}^{\infty} 2^{(i+1)\left(n-\frac n{p_\infty}+\frac n{p(x_0)}-\frac n{u(x_0)}-m\right)}\notag\\
&\hspace{2em}\approx r^{n\left(\frac1{p(x_0)}-\frac1{u(x_0)}\right)}r^{\left(n-\frac n{p_\infty}+\frac n{p(x_0)}-m\right)}  2^{(J+1)\left(n-\frac n{p_\infty}+\frac n{p(x_0)}-\frac n{u(x_0)}-m\right)}\notag\\
&\hspace{2em}= r^{n\left(\frac1{p(x_0)}-\frac1{u(x_0)}\right)}(2^{J+1}r)^{\left(n-\frac n{p_\infty}+\frac n{p(x_0)}-\frac n{u(x_0)}-m\right)}r^{\frac n{u(x_0)}}\notag\\
&\hspace{2em}\lesssim r^{n\left(\frac1{p(x_0)}-\frac1{u(x_0)}\right)}\notag,
\end{align}
so the estimate \eqref{eq:final} also holds for the term $I\!I\!I_2$ of $I\!I\!I$, and this finishes the case $0<r<1$.\\
\underline{Case $r\geq1$:} To ensure that \eqref{eq:final} holds we use again Lemma \ref{lem:xinorm}, the estimate $2^{\nu n}(1+2^\nu2^ir)^{-m} \leq 2^{\nu(n-m)}(2^ir)^{-m}$ and the hypothesis \eqref{condition_m} on $m$ and get
\begin{align}
I\!I\!I &\lesssim r^{n\left(\frac1{p(x_0)}-\frac1{u(x_0)}\right)} \sum_{i=1}^{\infty} r^{\frac n{p_\infty}} (2^{i+1}r)^{\frac n {p'_\infty}}(2^{i+1})^{n\left(\frac1{p(x_0)}-\frac1{u(x_0)}\right)}(2^ir)^{-m}\notag\\
&\approx r^{n\left(\frac1{p(x_0)}-\frac1{u(x_0)}\right)}r^{\frac n{p_\infty}+\frac n{p'_\infty}-m} \sum_{i=1}^{\infty}2^{(i+1)\left(  \frac n{p'_\infty}+\frac n{p(x_0)}-\frac n{u(x_0)}-m\right)}\notag\\
&= r^{n\left(\frac1{p(x_0)}-\frac1{u(x_0)}\right)}r^{n-m} \sum_{i=1}^{\infty}2^{(i+1)\left( n- \frac n{p_\infty}+\frac n{p(x_0)}-\frac n{u(x_0)}-m\right)}\lesssim r^{n\left(\frac1{p(x_0)}-\frac1{u(x_0)}\right)}\notag.
\end{align}
\end{proof}
\begin{rem}\label{rem:FuXu}
As stated before, this convolution inequality is the basis for all the main results on the Triebel-Lizorkin-Morrey spaces which we prove in this paper, as a replacement for the vector valued Hardy-Littlewood maximal inequality \cite{FeffermannStein}. Part of the proof relies on ideas used in \cite{TangXu} and \cite{FuXu} to deal with a corresponding maximal inequality in the context of Triebel-Lizorkin-Morrey spaces, respectively with constant exponents or with variable $p$ and $u$, keeping the other exponents fixed. There is, however, a problem with the {\em proof} in \cite[Theorem 2.2]{FuXu}: the authors always used $\norm{\chi_{B_r(x)}}{L_\p(\Rn)}\approx r^{\frac n{p(x)}}$. Although this is correct in the case $0<r\leq 1$, for $r>1$ one has to use the second case in Lemma \ref{lem:xinorm}. Which in the end means that the mentioned maximal inequality was not proved.\\
Here we also do not prove such an inequality, but were able to correctly use Lemma \ref{lem:xinorm} in order to prove the convolution inequality in the above Theorem. Moreover, we could allow all exponents to vary.   
\end{rem}
\begin{rem}
We would like to discuss the somewhat complicated condition \eqref{condition_m} on $m$, namely
\begin{align*}
m>n+n\max\left(0,\sup_{x\in\Rn}\left(\frac1{p(x)}-\frac1{u(x)}\right)-\frac1{p_\infty}\right).
\end{align*} 
The theorem above holds also with the shorter condition
\begin{align}\label{condition_m2}
m>n(1+c_\infty(1/p))
\end{align}
instead, where $c_\infty(1/p)$ is the constant from \eqref{dfn:cinfty} with $g=1/p$. It is easily seen that condition \eqref{condition_m} is sharper than \eqref{condition_m2}. Both conditions on $m$ seem to have their advantages and disadvantages. It is clear that in the case of variable Lebesgue spaces, i.e. when $p(x)=u(x)$, condition \eqref{condition_m} recovers the condition $m>n$ of Lemma \ref{lem:Convinequality}. Condition \eqref{condition_m2} does not enjoy such a property.\\
On the other hand, \eqref{condition_m2} looks nicer and in the case $p(\cdot)=p$ constant it reduces to $m>n$ no matter how $u(\cdot)\geq p$ is chosen. The same behaviour can also be observed with condition \eqref{condition_m}. \\
To make our results more accessible we introduce the following abbreviation, which we shall use in the rest of the paper:
\begin{align}\label{eq:cpu}
	\cpu:=\max\left(0,\sup_{x\in\Rn}\left(\frac1{p(x)}-\frac1{u(x)}\right)-\frac1{p_\infty}\right).
\end{align}
It is easily seen that $\cpu=0$ if $\p=\u$ or when $\p=p$ constant.
\end{rem}

\section{2-microlocal Triebel-Lizorkin-Morrey spaces and their Peetre maximal function characterization}
We need admissible weight sequences and so called admissible pairs in order to define the spaces under consideration.
\begin{dfn}\label{dfn:admissiblePair}
A pair $(\check{\varphi}, \check{\Phi})$ of functions in $\Sn$ is called admissible if
\begin{align*}
\supp{\varphi}&\subset\{x\in\Rn: \frac12\leq|x|\leq2\}\text{ and }\supp\Phi\subset\{x\in\Rn:|x|\leq2\}\text{ with}
\end{align*}
\begin{align*}
|\varphi(x)|\geq c>0\text{ on }\{x\in\Rn:\frac35\leq|x|\leq\frac53\}\text{ and }
|\Phi(x)|\geq c>0\text{ on }\{x\in\Rn:|x|\leq\frac53\}.
\end{align*}
Further, we set $\varphi_j(x):=\varphi(2^{-j}x)$ for $j\geq1$ and $\varphi_0:=\Phi$. Then $(\varphi_j)_{j\in\N_0}\subset\Sn$ and it holds $\supp{\varphi_j}\subset\{x\in\Rn:2^{j-1}\leq|x|\leq2^{j+1}\}.$
\end{dfn}
\begin{dfn}\label{weight}
 Let $\alpha_1\leq\alpha_2$ and $\alpha\geq0$ be real numbers. The class of admissible weight sequences $\mgk$ is the collection of all sequences $\vek{w}=(w_j)_{j\in\N_0}$ of measurable functions $w_j$ on $\Rn$ such that
 \begin{itemize}
  \item[(i)] There exists a constant $C>0$ such that 
  \begin{align*}
   0<w_j(x)\leq C w_j(y)(1+2^j|x-y|)^\alpha\text{ for }x,y\in\Rn \text{ and }j\in\N_0; 
  \end{align*}
  \item[(ii)] For all $x\in\Rn$ and $j\in\N_0$ 
  \begin{align*}
   2^{\alpha_1}w_j(x)\leq  w_{j+1}(x)\leq2^{\alpha_2}w_j(x). 
  \end{align*}
 \end{itemize}
\end{dfn}
Now, we can give the definition of 2-microlocal Triebel-Lizorkin-Morrey spaces.
\begin{dfn}\label{def:TLM}
 Let $(\varphi_j)_{j\in\N_0}$ be constructed as in Definition \ref{dfn:admissiblePair} and $\vek{w}\in\mgk$ be admissible weights. Let  $p,q\in\Plog$ and $u\in\mathcal{P}(\Rn)$ with $0<p^-\leq p(x)\leq u(x)\leq \sup u < \infty$ and  $q^-,q^+\in(0,\infty)$. Then
 \begin{align*}
  \Mufwpqx&:=\left\{f\in\SSn:\norm{f}{\Mufwpqx}<\infty\right\}
  \intertext{where}
  \norm{f}{\Mufwpqx}&:=\norm{(w_j(\varphi_j\hat{f})^\vee)_j}{\Muplq}  \\
	&:=\norm{\left(\sum_{j=0}^\infty|w_j(\varphi_j\hat{f})^\vee|^\q\right)^{1/\q}}{\Mup}.
 \end{align*}
\end{dfn}
\begin{rem}
\begin{itemize}
\item[(i)] Some properties for $\Mufwpqx$ directly carry over from the spaces they are built up from. So it is easily seen that $\Mufwpqx$ are quasi-normed spaces. In particular, the verification that $\norm{f}{\Mufwpqx}=0$ implies $f=0$ almost everywhere can be reduced to the observation that
\begin{align*}
\norm{f}{\Mufwpqx}=0\Leftrightarrow \norm{f}{\Fwpxqx}=0\Leftrightarrow f=0\ a.e..
\end{align*} 
Furthermore, since $L_\p(\Rn)$ are normed spaces for $p\geq1$, we have that $\Mup$ are normed spaces if $\p\geq1$. And since also $\ell_\q$ spaces are normed in the case $\q\geq1$, we finally get that $\Mufwpqx$ are normed spaces when $\min(p^-,q^-)\geq1$.\\
\item[(ii)] On the other hand, to show that
\begin{align}\notag
\Sn\hookrightarrow\Mufwpqx\hookrightarrow\SSn
\end{align}
we refer to our forthcoming paper \cite{PartII} on this topic. The proofs come from the arguments leading to the atomic and molecular characterizations of the spaces $\Mufwpqx$, constituting an approach different from the one classically used, see e.g. the proof of \cite[(2.3.3/1)]{Tri83}. 
\item[(iii)] It is known that $\Sn$ is not dense even in the constant exponents Morrey space $M^u_p(\Rn)$ when $p<u$, see \cite[Proposition 2.7]{Rosenthal}. So we can not expect $\Sn$ to be dense in $\Mufwpqx$ in general.
\end{itemize}
\end{rem}
Our main results in this section are the Peetre maximal function characterization of these spaces and the consequence that their definition does not depend on the admissible pair considered, up to equivalent quasi-norms.

To that end, we start by defining the Peetre maximal function.
It was introduced by Jaak Peetre in
\cite{PeetreArt}. Given a system $(\psi_j)_{j\in\N_0}\subset\Sn$ we set
$\Psi_j=\hat\psi_j\in\Sn$ and define the Peetre maximal
function of $f\in\SSn$ for every $j\in\N_0$ and $a>0$ as
\begin{align}\notag
    (\Psi_j^*f)_a(x):=\sup_{y\in\R^n}\frac{|(\Psi_j\ast
    f)(y)|}{1+|2^j(y-x)|^a},\quad x\in\R^n.
\end{align}
We start with two given functions $\psi_0,\psi_1\in\Sn$ and define $\psi_j(x):=\psi_1(2^{-j+1}x)$, for $x\in\R^n$ and
    $j\in\N\setminus\{1\}$.
Furthermore, for all $j\in\N_0$ we write, as mentioned, $\Psi_j=\hat{\psi_j}$.\\
Now, we state the main theorem.
\begin{thm}\label{thm:lm}
    Let $\vek{w}=(w_j)_{j\in\N_0}\in\mgk$. Assume $p,q\in\Plog$ and $u\in\mathcal{P}(\Rn)$ with $0<p^-\leq p(x)\leq u(x)\leq \sup u < \infty$ and  $q^-,q^+\in(0,\infty)$. Let $R\in\N_0$ with $R>\alpha_2$ and let further $\psi_0,\psi_1$ belong to $\Sn$ with
    \begin{align}
        D^\beta\psi_1(0)=0,\quad\text{ for
        }0\leq|\beta|< R,\label{LokaleMittelMomentCond}
    \end{align}
    and
    \begin{align}
        |\psi_0(x)|&>0\quad\text{on}\quad\{x\in\R^n:|x|\leq k\varepsilon\},\label{LokaleMittelTauber1,5}\\
        |\psi_1(x)|&>0\quad\text{on}\quad\{x\in\R^n:\varepsilon \leq |x| \leq 2k\varepsilon\}\label{LokaleMittelTauber2,5}
    \end{align}
    for some $\varepsilon>0$ and $k\in(1,2]$.\\
    For  $a>n\left(\frac{1}{\min(p^-,q^-)}+\cpu\right)+\alpha$ we have that
        \begin{align}\label{eq:lm}
            \norm{f}{\Mufwpqx}\approx\norm{((\Psi_j\ast
        f)w_j)_j}{\Muplq}\approx\norm{((\Psi_j^*f)_aw_j)_j}{\Muplq}
    \end{align}
    holds for all $f\in\SSn$.
\end{thm}
The proof relies on \cite{Rychkov} and \cite{Ull12} and it will briefly be given after some preparations. We outline the necessary changes to the proof in \cite{Kempka09} (cf. also \cite{AC16}, where the version with the $k$ was introduced), which are small, since the proof technique works mainly with pointwise estimates.

Notice that Theorem \ref{thm:lm} contains the conclusion that the 2-microlocal Morreyfied spaces of variable exponents given in Definition \ref{def:TLM} are independent of the admissible pair considered.

Before proving the theorem we state some useful lemmas.

The first one essentially states that the spaces $\Mup$ satisfy the lattice property, this being an easy consequence of the corresponding property for the spaces $L_\p$.
\begin{lem}
Let $f$ and $g$ be two measurable functions with $0\leq f(x)\leq g(x)$ for a.e. $x\in\Rn$.
Then it holds
\begin{align*}
\norm{f}{\Mup}\leq\norm{g}{\Mup}.
\end{align*}
\end{lem}
The next easy lemma shows some homogeneity in the exponents of the spaces.
\begin{lem}\label{lem:ttrick}
Let $p,q,u\in\mathcal{P}(\Rn)$ with $p\leq u$ and $0<t<\infty$. Then for any sequence $(f_\nu)_{\nu\in\N_0}$ of measurable functions it holds
\begin{align*}
\norm{(|f_\nu|^t)_\nu}{M_{\frac\p t}^{\frac\u t}(\ell_{\frac \q t})}&=\norm{\left(\sum_{\nu=0}^\infty|f_\nu|^\q\right)^{t/\q}}{M_{\frac\p t}^{\frac\u t}(\Rn)}\\
&=\norm{(f_\nu)_\nu}{\Muplq}^t,
\end{align*}
with the usual modification every time $q(x)=\infty$.
\end{lem}

The proof of the next lemma works exactly as in \cite[Lemma 4.2]{Kempka09} using now the lattice property of $\Mup$.
\begin{lem}\label{lem:Hardy}
 Let $p,q,u\in\mathcal{P}(\Rn)$ with $p\leq u$. For any sequence $(g_j)_{j\in\N_0}$ of non negative measurable functions we denote, for $\delta>0$,
 \begin{align*}
  G_k(x)=\sum_{j=0}^\infty2^{-|k-j|\delta}g_j(x), \quad x \in \Rn, \quad k\in \N_0.
 \end{align*}
 Then it holds
 \begin{align*}
  \norm{(G_k)_k}{\Muplq}\leq c(\delta,q)\norm{(g_j)_j}{\Muplq},\text{ where}\\
  c(\delta,q)=\max\left(\sum_{j\in\mathbb{Z}}2^{-|j|\delta},\left[\sum_{j\in\mathbb{Z}}2^{-|j|\delta q^-}\right]^{1/q^-}\right).
 \end{align*}
\end{lem}
Now, we are ready to prove Theorem \ref{thm:lm}. As with the non-Morreyfied spaces, we divide the proof into two separate assertions, from which we can then conclude. The first one compares two different Peetre maximal functions with two sets of start functions $\psi_0,\psi_1\in\Sn$ and $\phi_0,\phi_1\in\Sn$. As before we introduce the dilations
\begin{equation*}
\psi_j(x)=\psi_1(2^{-j+1}x)\text{ and }\phi_j(x)=\phi_1(2^{-j+1}x), \quad x \in \Rn, \quad j\in \N\setminus\{1\}, 
\end{equation*}\\[-1mm]
as well as $\Psi_j= \hat{\psi}_j\text{ and }\Phi_j=\hat{\phi}_j\text{ for all }j\in\N_0$.\\
\begin{thm}\label{thm:ComparisonLM}
Let $\vek{w}=(w_j)_{j\in\N_0}\in\mgk$ and $p,q,u\in\mathcal{P}(\Rn)$ with $p\leq u$. Let $a>0$ and $R\in\N_0$ with $R>\alpha_2$. Let further $\psi_0, \psi_1$ belong to $\Sn$ with
    \begin{align*}
        D^\beta\psi_1(0)=0,\quad0\leq|\beta|<R
    \end{align*}
    and $\phi_0, \phi_1$ belong to $\Sn$ with
    \begin{align*}
        |\phi_0(x)|&>0\quad\text{on}\quad\{x\in\R^n:|x|\leq k\varepsilon\},\\
        |\phi_1(x)|&>0\quad\text{on}\quad\{x\in\R^n:\varepsilon\leq |x|\leq 2k\varepsilon\}
    \end{align*}
		for some $\varepsilon>0$ and $k\in(1,2]$.
    Then
    \begin{align*}
        \norm{((\Psi_j^*f)_aw_j)_j}{\Muplq}\leq
        c\norm{((\Phi_j^*f)_aw_j)_j}{\Muplq}.
    \end{align*}
    holds for every $f\in\SSn$.
\end{thm}
\begin{proof}
The proof follows the lines of the proof of \cite[Theorem 3.6]{Kempka09} up to equation (21) of that paper, except that, because of the $k$ considered above, a different sequence $(\lambda_j)_{j\in\N_0}$ should be used there, similarly as in the proof of \cite[Theorem 3.1]{AC16}. Since this was not detailed in the latter paper, for the convenience of the reader we explain here how the modified $(\lambda_j)_{j\in\N_0}$ to be used in the proof of \cite[Theorem 3.6]{Kempka09} is constructed.

Consider a function $\Theta_0 \in \Sn$ such that for some fixed $\delta_1\in[0,1)$
\begin{align*}
\Theta_0(x)&=1 \quad \mbox{if }\; |x|\leq 1+\delta_1\text{ and}\\
\Theta_0(x)&=0 \quad \mbox{if }\; |x|\geq 1+\delta_2
\end{align*}
for some fixed $\delta_2\in(\delta_1,1]$, which is radially strictly decreasing for $|x|\in[1+\delta_1,1+\delta_2]$. Defining
$$\Theta_1(x):=\Theta_0(2^{-1}x)-\Theta_0(x)\text{, for } \quad x \in \Rn,$$
one has that ${\rm supp}\, \Theta_1 \subset \{ x \in \Rn: 1+\delta_1 \leq |x| \leq 2(1+\delta_2)\}$
(actually, $\,\Theta_1(x)=0\,$ iff $\,|x|\leq 1+\delta_1 \, \text{ or } \, |x|\geq 2(1+\delta_2)$). Further, we set
\begin{align*}
\Theta_j(x):=\Theta_1(2^{-j+1}x)&=\Theta_0(2^{-j}x)-\Theta_0(2^{-j+1}x),\quad x\in \Rn,\;\; j \in \N\setminus\{1\}\\
\text{and obtain}\quad\sum_{j=0}^\infty \Theta_j(x)&=1 \quad \text{for all }\; x \in \Rn.
\end{align*}

For each $k\in(1,2]$, we are going to consider such a function $\Theta_0$ for $\delta_1$ and $\delta_2$ above chosen in such a way that $k=\frac{1+\delta_2}{1+\delta_1}$ (notice that the range of $\frac{1+\delta_2}{1+\delta_1}$,  for $\delta_1, \delta_2$ as mentioned, is precisely the set $(1,2]$) and then define the sequence of functions $\lambda_j\in\Sn$ in the following way:
$$\begin{array}{ll}
j=0: & \lambda_0(x):=0 \quad \mbox{if }\; |x|\geq k\varepsilon;\\
     & \lambda_0(x):=\frac{\Theta_0\left(\frac{1+\delta_1}{\varepsilon}x\right)}{\phi_0(x)} \quad \mbox{if }\; |x|<k\varepsilon.\\[3mm]
j\in\N: & \lambda_j(x):=0 \quad \mbox{if }\; |2^{-j+1}x|\leq \varepsilon \; \text{ or } \; |2^{-j+1}x|\geq 2k\varepsilon;\\
        & \lambda_j(x):=\frac{\Theta_j\left(\frac{1+\delta_1}{\varepsilon}x\right)}{\phi_j(x)} \quad \mbox{if }\; \varepsilon<|2^{-j+1}x|<2k\varepsilon.
\end{array}$$
Notice that this makes sense, in view of the hypotheses on $\phi_0, \phi_1$. Notice also that, by construction,
\begin{equation}
{\rm supp}\, \lambda_0 \subset \{ x \in\Rn: |x| \leq k \varepsilon \}, \quad {\rm supp\,} \lambda_1 \subset \{ x \in \Rn: \varepsilon \leq |x| \leq 2k\varepsilon \},
\label{eq:*}
\end{equation}
\begin{equation}
\lambda_j(x) = \lambda_1(2^{-j+1}x), \quad x \in \Rn, \;\; j\in\N,
\label{eq:**}
\end{equation}
and
\begin{equation}
\sum_{j=0}^\infty \lambda_j(x) \phi_j(x) = 1, \quad x \in\Rn.
\label{eq:***}
\end{equation}

This corresponds to (17), (16) and (15) in the proof of \cite[Theorem 3.6]{Kempka09}. One gets exactly the mentioned expressions in that paper by particularizing $k$ to $2$ and by making our $\varepsilon>0$ here equal to half the $\varepsilon>0$ used there. We can then essentially follow the remainder of the proof of \cite[Theorem 3.6]{Kempka09}, with the more general sequence $(\lambda_j)_{j\in\N_0}$ considered above, up to the derivation of formula (21) in that paper, namely
\begin{align*}
(\Psi_\nu^*f)_a(x)w_\nu(x)\leq
        c\sum_{j=0}^\infty2^{-|j-\nu|\delta}(\Phi_j^*f)_a(x)w_j(x), \quad x \in \Rn,
\end{align*} 
with $\delta:=\min(1,R-\alpha_2)$. Now, we apply Lemma \ref{lem:Hardy} and the theorem is proved.
\end{proof}
Now for the second assertion needed in order to prove Theorem \ref{thm:lm}. We use the same notation introduced before.
\begin{thm}\label{thm:boundednessLM}
Let $\vek{w}=(w_j)_{j\in\N_0}\in\mgk$. Assume $p,q\in\Plog$ and $u\in\mathcal{P}(\Rn)$ with $0<p^-\leq p(x)\leq u(x)\leq \sup u < \infty$ and  $q^-,q^+\in(0,\infty)$. Let $\psi_0,\psi_1$ belong to $\Sn$ with
    \begin{align*}
        |\psi_0(x)|&>0\quad\text{on}\quad\{x\in\R^n:|x|\leq k\varepsilon\},\\
        |\psi_1(x)|&>0\quad\text{on}\quad\{x\in\R^n:\varepsilon\leq |x|\leq 2k\varepsilon\}
    \end{align*}
		for some $\varepsilon>0$ and $k\in(1,2]$.
		
    \noindent If $a>n\left(\frac{1}{\min(p^-,q^-)}+\cpu\right)+\alpha$, then
        \begin{align}\label{eq:boundednessLM}
            \norm{((\Psi_j^*f)_aw_j)_j}{\Muplq}\leq c\norm{((\Psi_j\ast f)w_j)_j}{\Muplq}
        \end{align}
    holds for all $f\in\SSn$.
\end{thm}
\begin{proof}
As in the last proof, we find functions $\lambda_j \in \Sn$, $j\in\N_0$, satisfying \eqref{eq:*}, \eqref{eq:**} and (instead of \eqref{eq:***})
$$\sum_{j=0}^\infty \lambda_j(x) \psi_j(x) = 1, \quad x \in\Rn.$$
So, now we are using $\psi_j$ in place of $\phi_j$ of the preceding theorem, observing that $\psi_0, \psi_1$ now satisfy the conditions assumed for $\phi_0, \phi_1$ there. With the $\lambda_j$'s above replacing the functions with the same names in the proof of \cite[Theorem 3.8]{Kempka09}, we can essentially follow that proof up to the derivation of formula (27) of that paper, where we can also replace the $a$ by the $M$ at our disposal.

For the next key estimate in this proof we prefer to follow now \cite{Ull12}, as the original approach of Rychkov \cite{Rychkov}, followed in \cite{Kempka09}, is known to have some issues solved in subsequent papers by the same author. This is well reported by Ullrich \cite{Ull12}, to which we refer the reader. So, adapting Subset 1.2 in the proof \cite[Theorem 2.6]{Ull12} to our context, we get, for every $t>0$ and every $N \in \N_0$,
$$|(\Psi_\nu\ast f)(y)|^t \leq c \sum_{j=0}^\infty 2^{-jNt}\, 2^{(j+\nu)n} \int_\Rn \frac{|(\Psi_{j+\nu}\ast f)(z)|^t}{(1+|2^\nu(y-z)|^N)^t}\, dz,$$
for all $y\in\Rn$,  where $c$ is independent of $f\in\SSn$, $\nu \in \N_0$ and $y\in\Rn$. This corresponds to \cite[(2.48)]{Ull12}. Then we can get, similarly as \cite[(2.66)]{Ull12}, for $0<a \leq N$, that
\begin{equation}
((\Psi^\ast_\nu f)_a(x))^t \leq c \sum_{j=0}^\infty 2^{-jNt}\, 2^{(j+\nu)n} \int_\Rn \frac{|(\Psi_{j+\nu}\ast f)(z)|^t}{(1+|2^\nu(x-z)|^a)^t}\, dz,
\label{eq:2.66}
\end{equation}
for all $x\in\Rn$,  where here $c$ is independent of $f\in\SSn$, $\nu \in \N_0$ and $x\in\Rn$.

Now we proceed as in the proof of \cite[Theorem 3.1]{AC16}. However, since the details are not given in the latter paper, for the convenience of the reader we write them down here: multiplying formula \eqref{eq:2.66} by $(w_\nu(x))^t$ and using the easy estimate
$$\frac{1}{1+|2^\nu(x-z)|^a} \leq \frac{2^{ja}}{1+|2^{j+\nu}(x-z)|^a} \approx \frac{2^{ja}}{(1+|2^{j+\nu}(x-z)|)^a}$$
and the estimate $w_\nu(x) \lesssim 2^{-j\alpha_1}w_{j+\nu}(z)(1+|2^{j+\nu}(x-z)|)^\alpha,$
which is an easy consequence of Definition \ref{weight}, we obtain for $a>\alpha$ that
\begin{align}
((\Psi^\ast_\nu f)_a(x))^t (w_\nu(x))^t &\notag\\ 
&\hspace{-5em}\lesssim  \sum_{j=0}^\infty 2^{-j(N-a+\alpha_1)t}\, 2^{(j+\nu)n} \int_\Rn \frac{|(\Psi_{j+\nu}\ast f)(z)|^t (w_{j+\nu}(z))^t}{(1+|2^{j+\nu}(x-z)|)^{(a-\alpha)t}}\, dz \notag \\
&\hspace{-5em} =  \sum_{j=\nu}^\infty 2^{-(j-\nu)(N-a+\alpha_1)t} \big(\eta_{j,(a-\alpha)t}\ast (|\Psi_j \ast f|w_j)^t\big)(x),\label{pointwiseestimate}
\end{align}
where the constant involved in the estimate above is independent of $f\in\SSn$, $\nu \in \N_0$ and $x\in\Rn$.

So, up to now we obtained pointwise estimates which have nothing to do with the consideration of the parameters $p, q, u$. Consider now\\ $a>\alpha+n\left(\frac{1}{\min(p^-,q^-)}+\cpu\right)$, $N> a+|\alpha_1|$ and $t\in (0,\min(p^-,q^-))$ such that still $a>\alpha+n\left(\frac{1}{t}+\cpu\right)$. Then applying first Lemma \ref{lem:Hardy} with $\delta:=(N-a+\alpha_1)t$ and afterwards Theorem \ref{thm:MorreyHardy} with $m:=(a-\alpha)t$ we get from \eqref{pointwiseestimate}
				\begin{align*}
				\norm{\left(((\Psi_\nu^*f)_a w_\nu)^t\right)_\nu}{M_{\frac\p t}^{\frac\u t}(\ell_{\frac\q t})}
				&\lesssim \norm{\left(\eta_{j,(a-\alpha)t}\ast (|(\Psi_{j}\ast
        f)|w_{j})^t\right)_j}{M_{\frac\p t}^{\frac\u t}(\ell_{\frac\q t})}\\
				&\lesssim \norm{\left((|(\Psi_{j}\ast
        f)|w_{j})^t\right)_j}{M_{\frac\p t}^{\frac\u t}(\ell_{\frac\q t})}.
				\end{align*}
The conclusion \eqref{eq:boundednessLM} now follows by application of Lemma \ref{lem:ttrick}.							
\end{proof}

{\em Proof of Theorem \ref{thm:lm}.}
Assume the hypotheses stated. 
The second estimate in \eqref{eq:lm} comes immediately from Theorem \ref{thm:boundednessLM} and the easy estimate $(\Psi^*_jf)_a(x)\geq|(\Psi_j\ast f)(x)|$.
Considering, besides $\psi_0, \psi_1$ as in Theorem \ref{thm:lm}, also $\phi_0, \phi_1$ satisfying the formulas corresponding to \eqref{LokaleMittelMomentCond}, \eqref{LokaleMittelTauber1,5} and \eqref{LokaleMittelTauber2,5}, possibly with different $\varepsilon$ and $k$ than the ones used for $\psi_0, \psi_1$, the estimate 
$$\norm{((\Psi_j\ast f)w_j)_j}{\Muplq} \approx \norm{((\Phi_j\ast f)w_j)_j}{\Muplq}$$
follows easily by application of Theorems \ref{thm:ComparisonLM} and \ref{thm:boundednessLM}. So, the proof of Theorem \ref{thm:lm} will be finished if one shows that each construction of $(\check{\varphi}_j)_{j \in \N_0}$ as in Definition \ref{def:TLM} coincides with one of the possibilities for $(\Phi_j)_{j\in\N_0}$ above (for some choice of $\varepsilon>0$ and $k\in(1,2]$). This is indeed the case: given an admissible pair $(\check{\varphi}, \check{\Phi})$ according to Definition \ref{dfn:admissiblePair}, define $\phi_0:=\varphi_0(-\,\cdot):=\Phi(-\,\cdot)$ and $\phi_1:=\varphi_1(-\,\cdot):=\varphi(2^{-1}(-\,\cdot))$; it is easy to see that this produces Schwartz functions satisfying \eqref{LokaleMittelMomentCond}, \eqref{LokaleMittelTauber1,5} and \eqref{LokaleMittelTauber2,5} with $\phi_0, \phi_1$ instead of $\psi_0, \psi_1$ and for $\varepsilon = \frac{6}{5} > 0$ and $k = \frac{25}{18} \in (1,2]$; and that, moreover, $\check{\varphi}_j = \Phi_j$ for all $j\in\N_0$.\qed

The last proof gives that the definition of the spaces $\Mufwpqx$ is independent of the chosen admissible system, therefore the spaces are well defined. For further properties we refer to our forthcoming paper \cite{PartII} where decompositions with atoms and molecules in this scale are treated.

\small{
}

\end{document}